\documentclass[11pt]{article}

\usepackage{amsthm,amssymb,amsmath}
\usepackage{comment,authblk}
\usepackage{url,fullpage}

\newtheorem{theorem}{Theorem}[section]
\newtheorem{lemma}[theorem]{Lemma}
\newtheorem{corollary}[theorem]{Corollary}

\theoremstyle{definition}
\newtheorem{definition}[theorem]{Definition}
\newtheorem{example}[theorem]{Example}

\theoremstyle{remark}
\newtheorem{remark}[theorem]{Remark}

\renewcommand{\SS}{\ensuremath{\mathcal{S}}}
\newcommand{\K}{\ensuremath{\mathcal{K}}}
\newcommand{\eff}{\ensuremath{\mathbb{F}}}
\newcommand{\B}{\ensuremath{\mathcal{B}}}
\renewcommand{\P}{\ensuremath{\mathcal{P}}}
\newcommand{\C}{\ensuremath{\mathcal{C}}}
\newcommand{\M}{\ensuremath{\mathcal{M}}}
\newcommand{\prob}{\ensuremath{\mathrm{Pr}}}
\newcommand{\R}{\ensuremath{\mathcal{R}}}
\newcommand{\STS}{\ensuremath{\mathsf{STS}}}
\newcommand{\SQS}{\ensuremath{\mathsf{SQS}}}
\newcommand{\BIBD}{\ensuremath{\mathsf{BIBD}}}

\begin{document}

\title{A Network Reliability Approach to the Analysis of Combinatorial Repairable Threshold Schemes}

\author{Bailey Kacsmar}
\author{Douglas R.\ Stinson%
\thanks{D.R.\ Stinson's Research is supported by  NSERC discovery grant RGPIN-03882.}}
\affil{David R.\ Cheriton School of Computer Science, University of Waterloo,
Waterloo, Ontario N2L 3G1, Canada}


\date{\today}

\maketitle

\begin{center}
\emph{Dedicated to Bimal Roy}
\end{center}

\vspace{.25in}

\begin{abstract}
A \emph{repairable threshold scheme} (which we abbreviate to  \emph{RTS}) 
is a  $(\tau,n)$-threshold scheme in which a subset of players 
can ``repair'' another player's share in the event that their share has been lost or corrupted.
This will take place without the participation of the dealer who set up the scheme.
The repairing protocol should not compromise the (unconditional) security of the threshold scheme.
\emph{Combinatorial repairable threshold schemes} (or \emph{combinatorial RTS})
 were recently introduced by Stinson and Wei \cite{stinson-wei}. In these schemes, 
``multiple shares'' are distributed to each player, as defined by a suitable
combinatorial design called the \emph{distribution design}. In this paper, 
we study the \emph{reliability} of these combinatorial repairable threshold schemes in a setting
where players may not be available to take part in a repair of a given player's share.
Using techniques from network reliability theory, we consider the probability of existence of an available
repair set, as well as the expected number of available repair sets, for various types of distribution
designs.
\end{abstract}


\section{Introduction to Combinatorial Repairability}

Suppose that $\tau$ and $n$ are positive integers
such that  $\tau \leq n$. 
Informally, a \emph{$(\tau,n)$-threshold scheme} is 
a method whereby a {\it dealer} chooses a {\it secret} and distributes a {\it share} to 
each of $n$ {\it players} (denoted by $P_1, \dots , P_n$) such that the following two properties
are satisfied:
\begin{description}
\item[reconstruction] Any subset of $\tau$ players can compute the
secret from the shares that they collectively hold, and 
\item[secrecy] No subset of $\tau-1$ players can determine any information about 
the secret.
\end{description}
{We call $\tau$  the {\it threshold} of the scheme.}

In this paper, we are only interested in schemes that are {\it unconditionally secure}.
That is, all security results are valid against adversaries with unlimited computational power.

The efficiency of secret sharing is often measured in terms of the
\emph{information rate} of the scheme, which is defined to be the ratio
$\rho = \log _2 | \K | / \log _2 |  \SS |$
(where $\SS$ is the set of all possible shares and $\K$ is the set of all
possible secrets). That is, the information rate is the ratio of the size of the secret to the
size of a share.
For a threshold scheme, a fundamental result
states that $\rho \leq 1$. 

We briefly describe a standard construction for threshold schemes with optimal
information rate, namely, the classical
Shamir threshold scheme \cite{Shamir}. 
The construction takes place over a finite field $\eff_Q$, where $Q \geq n+1$.
\begin{enumerate}
\item In the {\bf Initialization Phase},  
the dealer, denoted by $D$, chooses $n$ distinct, 
non-zero elements of $\eff_Q$, denoted $x_i$, 
$1 \leq i \leq n $. The values $x_i$ are public.
For $1 \leq i \leq n$, $D$ gives the value $x_i$ to player $P_i$.  
\item In the {\bf Share Distribution} phase, 
$D$ chooses a secret \[K = a_0 \in \eff_Q.\]
Then $D$ secretly chooses (independently and uniformly at random)
 \[a_{1}, \dots, a_{\tau-1} \in \eff_Q.\]
Finally, for $1 \leq i \leq n$, 
$D$ computes the share $y_i= a(x_i)$, where 
\[a(x) = \sum _{j=0}^{\tau-1} a_j\, x^j,\]
and gives it to player $P_i$.
\end{enumerate}
Reconstruction is easily accomplished using the Lagrange interpolation formula
(see, e.g., \cite[\S11.5.1]{CTAP}).

The problem of {\it share repairability} has been considered by 
several authors in recent years (see Laing and Stinson \cite{laing-stinson} for a  survey on this topic). 
The problem setting is that a 
certain player $P_{\ell}$ (in a $(\tau,n)$-threshold scheme, say) loses their share. The goal is 
to find a ``secure'' protocol involving $P_{\ell}$ and a subset of the other
players that allows the missing share $y_{\ell}$ to be reconstructed.
(Of course the dealer could simply re-send the share to $P_{\ell}$, 
but we are considering a setting where the dealer is
no longer present in the scheme after the initial setup.)
In general, we will assume secure pairwise channels linking
pairs of players.

A combinatorial solution to this problem was proposed by Stinson and Wei \cite{stinson-wei}.
These schemes are termed \emph{combinatorial RTS}.
The construction is based on an old technique from \cite[Theorem 1]{BL},
namely, giving each player a subset of shares from an underlying threshold  scheme called a 
\emph{base scheme}\footnote{Actually, there are situations where there are efficiency 
advantages to use
a \emph{ramp scheme} as the base scheme, instead of a threshold scheme. This is addressed
in some detail in \cite{stinson-wei}. However, for the purposes of repairability,
it is irrelevant if we use a ramp scheme, as opposed to a threshold scheme. So we do not
discuss the use of ramp schemes in this paper.}

Suppose the base scheme is an $(\sigma,m)$-threshold scheme, say a Shamir scheme,
implemented over a finite field $\eff_Q$. We then give each 
player a certain subset of $d$ of the $m$ shares. A set system (or {\it design}) consisting
of $n$ blocks of size $d$, defined on a set of $m$ points, will
be used to do this. This design is termed the {\it distribution design}.

We will call the shares of the base
$(\sigma,m)$-threshold scheme {\it subshares}. Each share in the resulting
$(\tau,n)$-threshold scheme, which we call the \emph{expanded scheme}, consists of $d$ subshares.
Suppose the shares in the base scheme are denoted $s_1,\dots, s_m$, and suppose that the
points in the distribution design are denoted $1 , \dots , m$. Each player $P_i$ corresponds to a block
$B_i$ of the distribution design. For each point $x \in B_i$, the player $P_i$ is given the subshare
$s_x$. The points in a block are indices of subshares received by a given player. The blocks are
public information, while the values of the shares and subshares are secret.

We need to ensure that the relevant threshold
property is satisfied for the expanded threshold scheme. 
We also need to be able to repair the share of any player in the expanded scheme
by appropriately choosing a certain set of other players, who will then send
appropriate subshares to the player whose share is being repaired.

Let the blocks in the distribution design be denoted $B_1, \dots , B_n$ and let $X$
denote the set of $m$ points on which the design is defined. 
The desired threshold property for the expanded scheme will be satisfied
provided that the following two conditions hold in the distribution design:
\begin{equation}
\label{eq1}
\text{the union of any $\tau$ blocks contains at least $\sigma$ points}
\end{equation} and
\begin{equation}
\label{eq2}
\text{the union of any $\tau-1$ blocks contains at most $\sigma -1$ points.}
\end{equation}

Summarizing, we have the  following theorem from \cite{stinson-wei}.

\begin{theorem}
\label{threshold.thm}
Suppose $(X,\B)$ is a distribution design with $|X| = m$ and $|\B| = n$.
Let $\tau$ and $\sigma$ be positive integers and
suppose (\ref{eq1}) and (\ref{eq2}) are satisfied for the given distribution design.
Then, if we use a base $(\sigma,m)$-threshold scheme  in conjunction with the
given distribution design, we obtain an expanded $(\tau,n)$-threshold scheme.
\end{theorem}

\subsection{Repairing a Share} 
 
Now, suppose we want to repair the share for a player $P_{\ell}$ 
corresponding to the block $B_{\ell}$. For each point $x \in B_{\ell}$, we  find another block
that contains $x$. 
The corresponding player can send the subshare $s_x$ corresponding to $x$
to $P_{\ell}$. We illustrate the technique with an example.

\begin{example}\label{STS9.ex}
Suppose we start with a $(9,3,1)$-$\BIBD$ (an affine plane of order $3$), which has
$n=12$ blocks of size $d=3$. There are $m=9$ points in the design.
We associate a block of the design with each player:
	\begin{align*}
	{P_1} &\leftrightarrow B_1 = {\{1, 2, 3\}} \quad & {P_2} &\leftrightarrow B_2 = {\{4, 5, 6\}}	\quad & {P_3} &\leftrightarrow B_3 = {\{7, 8, 9\}}	\\ 
	{P_4} &\leftrightarrow B_4 = {\{1, 4, 7\}}	\quad & {P_5} &\leftrightarrow B_5 = {\{2, 5, 8\}} \quad & {P_6} &\leftrightarrow B_6 = {\{3, 6, 9\}}	\\ 
	{P_7} &\leftrightarrow B_7 = {\{1, 5, 9\}}	\quad & {P_8} &\leftrightarrow B_8 = {\{2, 6, 7\}} \quad & {P_9} &\leftrightarrow B_9 = {\{3, 4, 8\}} \\
 {P_{10}} &\leftrightarrow B_{10} = {\{1, 6, 8\}} \quad & {P_{11}} & \leftrightarrow B_{11} = {\{2, 4, 9\}} \quad & {P_{12}} &\leftrightarrow B_{12} = {\{3, 5, 7\}}
	\end{align*}
Each player gets $d=3$ shares from a $(5,9)$-threshold scheme, as specified by the
	associated block.
This threshold scheme has nine shares, denoted ${s_1}$, $\dots$, ${s_9}$.
Each block lists the indices of shares held by a given player;  
thus $P_1$ has the shares $s_1$, $s_2$ and $s_3$. Each block contains three points and the union of
any two blocks contains at least five points. Thus (\ref{eq1}) and (\ref{eq2}) are satisfied for
$\tau = 2$ and $\sigma = 5$ and therefore the expanded scheme is a $(2,12)$-threshold scheme.
 
 Now suppose $P_1$ wishes to repair their share. $P_1$ requires the subshares $s_1$, $s_2$ and $s_3$.
The subshare $s_1$ can be obtained from $P_4, P_7$ or $P_{10}$;
the subshare $s_2$ can be obtained from $P_5, P_8$ or $P_{11}$;
and the subshare $s_3$ can be obtained from $P_6, P_9$ or $P_{12}$.
\hfill $\blacksquare$
\end{example}

In general, it is not a requirement that the $d$ subshares are obtained from $d$ different 
blocks. For example, it could happen that $d=3$, one block contributes two
subshares, and one block contributes one subshare during the repairing process.
See Section \ref{t.sec} for further discussion of this idea.

It is quite simple to analyze the security of combinatorial repairability. 
The main point to observe is that the information collectively held by any subset of players (after the 
repairing protocol is completed) consists only of their shares in the expanded scheme.
They did not obtain any information collectively that they did not already possess
before the execution of the repairing protocol.
So, it is immediate that a set $\tau-1$ players cannot compute the secret after the
repairing of a share occurs.

The paper \cite{stinson-wei} provides several constructions for combinatorial RTS.
Different distribution designs are studied and analyzed according to various metrics.
Here, we are only interested in repairability properties, so we do not address these
other metrics.

\subsection{Reliability}
\label{rel.sec}

Given a player $P_{\ell}$ in a combinatorial RTS, a subset of players that can repair $P_{\ell}$'s share is called 
a \emph{repair set} for  $P_{\ell}$. A repair set $\P$ for a player $P_{\ell}$ 
is \emph{minimal} if no proper
subset of $\P$ is a repair set for $P_{\ell}$.
In Example \ref{STS9.ex}, there are $3^3 = 27$ minimal repair sets for any given player.

We are interested in studying the situation where some of the players might not be available when 
asked to provide a subshare to repair another player's share. We will make the assumption that any player
is available with a fixed probability $p$ (and therefore unavailable with probability $1-p$).
We also assume that the availability of any player is independent of the availability
of any other player. A repair set is \emph{available} if every player in the set is available.

In the above setting, we can ask two basic questions for a given player associated with a given distribution design:
\begin{enumerate}
\item What is the probability $\R(p)$ that there is at least one available repair set?
\item What is the expected number $E(p)$ of available minimal repair sets?
\end{enumerate}

We illustrate these concepts by considering the distribution design presented in  Example \ref{STS9.ex}.

\begin{example}
\label{STS9-2.ex}
There is an available repair set for $P_1$ 
if and only if 
\begin{itemize}
\item at least one of $P_4, P_7$ or $P_{10}$ is available,
\item at least one of $P_5, P_8$ or $P_{11}$ is available, and 
\item at least one of $P_6, P_9$ or $P_{12}$ is available.
\end{itemize}
Therefore, for $P_1$,
\[ \R(p) = (1 - (1-p)^3)^3.\]
In fact, $\R(p)$ takes on the same value for any player in this RTS.

To compute the expected number of minimal repair sets, we observe that there are $27$ minimal repair sets,
each of which is available with probability $p^3$. By linearity of expectation, 
$E(p) = 27p^3$. Again, this value is the same for any player in the scheme.
\hfill $\blacksquare$
\end{example}

\subsection{Design Theory Definitions}

We now review some standard definitions and basic results from design theory.
Most of these results can be found in standard references such as \cite{HCD}.

\begin{definition} 
Suppose $2 \leq k < v$. A \textit{$(v,k,\lambda)$-balanced incomplete block design}, 
or \textit{$(v,k,\lambda)$-$\BIBD$}, is a design $(X, \B)$ such that: 
\begin{itemize}
  \item[1.] $|X| = v$,
  \item[2.] each block $B \in \B$ contains exactly $k$ points, and
  \item[3.] every pair of distinct points from $X$ is contained in exactly $\lambda$ blocks. 
\end{itemize}
\end{definition}

\begin{theorem}\label{repNum}
Every point in a $(v,k,\lambda)$-$\BIBD$ occurs in exactly 
\begin{center}
    $r=\frac{\lambda(v-1)}{k-1}$
\end{center}
blocks. The value $r$ is termed the \emph{replication number}.
\end{theorem}

\begin{theorem}
A $(v,k,\lambda)$-$\BIBD$ has exactly
\[b = \frac{vr}{k} = \frac{\lambda (v^2 -v)}{k^2 -k}\]
blocks of size $k$.
\end{theorem}

\begin{definition} 
A \textit{Steiner triple system}, or $\STS(v)$,  is a $(v,3,1)$-$\BIBD$. 
\end{definition}

\begin{theorem}
There exists an $\STS(v)$ if and only if $v \equiv 1, 3 \pmod  6$, $v \geq 7$.
\end{theorem}

\begin{definition}\label{def:tdesign}
A \emph{$t$-$(v,k,\lambda)$-design} is a design where: 
\begin{itemize}
    \item[1.] $|X|=v$,
    \item[2.] each block $B \in \B$ contains exactly $k$ points, and
    \item[3.] every set of $t$ points from the set $X$ occurs in exactly $\lambda$ blocks.
\end{itemize}
\end{definition}

\begin{definition}
A $3$-$(v,4,1)$-design is a \textit{Steiner quadruple system} of order $v$, denoted $\SQS(v)$. 
\end{definition}

\begin{theorem}
An $\SQS(v)$ exists if and  only if $v \equiv 2, 4  \pmod 6$.
\end{theorem}

\begin{theorem}\cite[Theorem II.4.8]{HCD}\label{thm:r_i}
The \emph{$i^{\text{th}}$ replication number}, denoted $r_i$, of a  $t$-$(v, k, 1)$-design is defined to be  
the number of blocks containing any given set of $i$ points. It is known that
\[r_i= \frac{ \lambda\binom{v-i}{t-i}}{\binom{k-i}{t-i}}, \] 
for $1 \leq i \leq t$. 
\end{theorem}

\begin{theorem} 
The number of blocks in a  $t$-$(v, k, 1)$-design is
\[b=\frac{\binom{v}{t}}{\binom{k}{t}}=\frac{vr_1}{k}.\]
\end{theorem}

\begin{definition}
An \emph{inversive geometry} is a $3$-$(n^d+1, n+1, 1)$-design, where $d \geq 2$. 
\end{definition}

\begin{theorem}
An inversive geometry
exsits for any $d \geq 2$ if $n$ is a prime power.
\end{theorem}

\subsection{Organization of the Paper}

The remaining sections of the paper are organized as follows.
In Section \ref{bibd.sec}, we study the reliability metrics for
BIBDs. In Section \ref{t.sec}, we turn to $t$-designs with $t >2$,
which have not previously been studied as distribution designs.
After addressing the possible thresholds that can be obtained,
we again consider the reliability metrics. Finally, Section \ref{sum.sec}
is a brief summary.

\section{Using BIBDs as Distribution Designs}
\label{bibd.sec}

Stinson and Wei \cite{stinson-wei} examined several types of BIBDs with $\lambda = 1$ 
for use as distribution designs in combinatorial RTS. They studied the thresholds of these
RTS as well as their efficiency with respect to storage, communication complexity
and computational complexity. In this section, we study the reliability of these RTS
using the measures defined in Section \ref{rel.sec}. 

Before proceeding further, we define some notation that will be used in the rest of the paper.

\begin{definition}
\label{rel.defn}
Suppose $(X,\B)$ is a  distribution design for a combinatorial RTS.
For any fixed block $B_i \in \B$, let $P_i$ be the corresponding 
player in the RTS. Further, for any $x_j \in B_i$, 
define \[ \C_j = \{ B \in \B \setminus \{ B_i\}: x_j \in B\}.\]
Finally, let $\P_j = \{ P_i: B_i \in \C_j\}$.
\end{definition}

\begin{example}
We refer to Examples  \ref{STS9.ex} and \ref{STS9-2.ex}.  
For the block $B_1 = \{1,2,3\}$, we have 
\begin{eqnarray*}
\C_1 &=& \{ B_4, B_7, B_{10}\} \\ 
\C_2 &=& \{B_5, B_8, B_{11}\} \\ 
\C_3 &=&\{B_6, B_9, B_{12}\}
\end{eqnarray*}
and therefore 
\begin{eqnarray*}
\P_1 &=& \{ P_4, P_7, P_{10}\} \\ 
\P_2 &=& \{P_5, P_8, P_{11}\} \\ 
\P_3 &=&\{P_6, P_9, P_{12}\}.
\end{eqnarray*}
\hfill $\blacksquare$
\end{example}

As in Section \ref{rel.sec}, we define $\R(p)$ to be the probability
that there is at least one available repair set for a given player.

\begin{theorem}\label{thm:existsBIBD}
Suppose $(X,\B)$ is a $(v,k,1)$-$\BIBD$ 
that is used as a distribution design for a combinatorial RTS, and let $P_i$ be any
player in the scheme. Then
\[\R(p)=(1-(1-p)^{r-1})^k .\]
\end{theorem}
\begin{proof}
Let the block corresponding to $P_i$ be $B_i = \{x_1, \dots , x_k\}$.
Consider  the sets $\C_j$, for $1 \leq j \leq k$, as defined in Definition \ref{rel.defn}. 
Clearly $|\C_j| = r-1$ for $1 \leq j \leq k$ and
$\C_j \cap \C_{j'} = \emptyset$ if $j \neq j'$. 

The probability that at least one player in $\P_j$ is available is $1-(1-p)^{r-1}$.
Then, since the sets $\P_j$ are disjoint, the probability that
at least one player in each $\P_j$ is available is $(1-(1-p)^{r-1})^k$.
\end{proof}

Now we consider the expected number of repair sets when using  a $(v,k,1)$-$\BIBD$ as a distribution design.


\begin{theorem}\label{thm:expectbibd}
Suppose $(X,\B)$ is a $(v,k,1)$-$\BIBD$ 
that is used as a distribution design for a combinatorial RTS, and let $P_i$ be any
player in the scheme. Then
\[E(p) = (r-1)^kp^k.\]
\end{theorem}
\begin{proof}
Let $B_i = \{x_1, \dots , x_k\}$.
The minimal repair sets are precisely the sets in $\P_1 \times \cdots \times \P_k$.
The number of minimal repair sets is therefore $(r-1)^k$.
The probability that a given minimal repair set is available is $p^k$.

Let the minimal repair sets be enumerated as $\M_1, \dots , \M_s$, where $s = (r-1)^k$.
For $1 \leq i \leq s$, let the random variable $\mathbf{X_i}$ be defined as 
\[\mathbf{X_i} = \begin{cases} 1, & \mbox{if $\M_i$ is available}  \\ 0, & \mbox{otherwise.} \end{cases}\]
Clearly $E[\mathbf{X_i}] = p^k$ for all $i$.
Define  \[\mathbf{X} = \mathbf{X_1} + \mathbf{X_2} + \dots + \mathbf{X_s};\]
then \[E[\mathbf{X}] = E[\mathbf{X_1}] + E[\mathbf{X_2}] + \dots + E[\mathbf{X_s}]\] by linearity of expectation.
Therefore,  \[E(p) = E[\mathbf{X}] = sp^k = (r-1)^kp^k.\]
\end{proof}

It would of course be possible to use a  $(v,k,\lambda)$-$\BIBD$ 
as a distribution design even if $\lambda > 1$. Unfortunately, 
there do not seem to be general formulas, analogous to Theorems \ref{thm:existsBIBD} 
and \ref{thm:expectbibd}, for these designs.

\section{Using $t$-Designs as Distribution Designs}
\label{t.sec}

It is also possible to use $t$-$(v,k,1)$-designs with $t>2$ as distribution designs. This idea
has not previously been discussed in the literature. One possible advantage over just using $2$-designs
is that blocks can intersect in more than one point, so a repair may be possible by contacting a smaller
number of other players. Since blocks in a $t$-$(v,k,1)$-design can intersect in up to $t-1$ points,
it follows that a repair can be carried out by contacting $\lceil\frac{k}{t-1}\rceil$ other players,
if they are available.

First, we determine the thresholds that can be achieved,
in particular, by Steiner quadruple systems and inversive geometries. 
Later in this section we analyze the reliability
of the RTS derived from them.

\subsection{Distribution Designs and Thresholds}\label{sect:TdistDes}

For a given distribution design, it is of interest to determine the thresholds
that can be realized in an expanded scheme. This involves choosing values for
$\tau$ and $\sigma$ in such a way that (\ref{eq1}) and (\ref{eq2}) are satisfied, and then
applying Theorem \ref{threshold.thm}.
We provide some results along this line in this section. We note that similar techniques were used 
in \cite{stinson-wei} for $2$-designs.

\begin{theorem}\label{thm:sqsasRTS}
An $\SQS(v)$ can be used as a distribution design to produce an RTS with threshold
$2$.
\end{theorem}
\begin{proof} 
Let $\tau = 2$ and $\sigma = 6$. It is clear that one block in an $\SQS(v)$ contains exactly four points.
Two blocks contain at least six points, because two blocks intersect in at most two points.
Therefore, (\ref{eq1}) and (\ref{eq2}) are satisfied when $\tau = 2$ and $\sigma = 6$, and
we obtain an expanded scheme with threshold $2$. 
\end{proof}

Now, we show how to construct RTS with threshold $3$ from certain $t$-designs.

\begin{theorem}\label{thm:get3threshold}
A $t$-$(v,k,1)$-design can be used as a distribution design to produce an RTS with threshold
$3$ if $k\geq 3t-2$.
\end{theorem}

\begin{proof}
Let $\tau = 2$ and $\sigma = 3k - 3(t-1)$.
Clearly the union of any two blocks contains at most $2k$ points.
Now consider three blocks. If any two of these blocks have $t-1$ points in common,
and these three intersections are disjoint, then the three blocks contain
$3k - 3(t-1)$ points, which is the minimum possible. In order for (\ref{eq1}) and (\ref{eq2}) to be satisfied,
we require $3k - 3(t-1) \geq 2k +1$, which is equivalent to $k\geq 3t-2$. If this inequality is
satisfied, then the expanded scheme has threshold $3$.
\end{proof}

More generally, we have the following result, which has a similar proof.

\begin{theorem}\label{thm:getanythreshold}
A $t$-$(v,k,1)$-design can be used as a distribution design to produce an RTS with threshold
$\tau$ if $k\geq \binom{\tau}{2}(t-1)+1 $.
\end{theorem}
\begin{proof}
Let $\sigma = \tau k- \binom{\tau}{2}(t-1)$.  Clearly $\tau-1$ blocks contain at most $(\tau-1)k$ points.
For a set of $\tau$ blocks, the minimum size of their union results when any two of them contain
$t-1$ common points, and these intersections are all disjoint.  So the union contains
at least $\tau k - \binom{\tau}{2}(t-1)$ points. In order for (\ref{eq1}) and (\ref{eq2}) to be satisfied,
we require $\tau k - \binom{\tau}{2}(t-1) \geq (\tau-1)k +1$, which is equivalent to 
$k\geq \binom{\tau}{2}(t-1)+1 $. If this inequality is
satisfied, then the expanded scheme has threshold $\tau$.
\end{proof}

The inversive geometries allow us to construct RTS with any desired threshold.
Taking $t=3$ in Theorem \ref{thm:getanythreshold}, we have the following corollary.

\begin{corollary}\label{cor:getanythreshold}
A $3$-$(v,k,1)$-design can be used as a distribution design to produce an RTS with threshold
$\tau$ if $k\geq \tau(\tau -1)+1 $.
\end{corollary}

\begin{remark}
In order to obtain $\tau = 3$, we require $k \geq 7$ in Corollary \ref{cor:getanythreshold};
to obtain $\tau = 4$, we require $k \geq 13$, etc.
\end{remark}

\subsection{Reliability}

In our analysis, to compute the reliability metrics for  repair sets, we employ the use of \textit{cutsets} from network reliability theory (see Colbourn \cite{colbourn} for basic results and terminology relating to network reliability).
When using BIBDs as distribution designs,  we were able to easily compute reliability formulas 
in Section \ref{bibd.sec} without the use of this methodology because the sets
$\C_j$ were disjoint. However, it is advantageous to use cutsets to 
analyze the reliability of the RTS constructed using distribution designs with $t \geq 3$.

In this section, for brevity, we will conflate the notion of players and blocks and express all our 
arguments in terms of blocks of the distribution design $(X,\B)$.


\begin{definition}
A \textit{cutset} for a block $B$ is a minimal subset of blocks $\B'$ 
such that a repair is not possible if all the blocks in $\B'$ are not available. 
A cutset \emph{fails} if every block in the cutset is not available.
\end{definition}

\begin{lemma}\label{def:failedC_i}
Let $B = \{x_1,\dots, x_k\}$ be a block in the distribution design.
Then the sets $\C_j$, for $1 \leq j \leq k$, are the  cutsets.
\end{lemma}

\begin{example}
Here are the blocks in an  $3$-$(8,4,1)$-design:
 		\begin{align*}
	 {A_1} &= {\{1,2,3,4\}} \quad 
	& {A_2} &= {\{5,6,7,8\}}	\\ 
	{B_1} &= {\{1,2,5,6\}}	\quad 
	& {B_2} &= {\{1,2,7,8\}}	\\ 
	 {B_3} &= {\{1,3,5,7\}} \quad
	& {B_4} &= {\{1,3,6,8\}}	\\
	 {B_5} &= {\{1,4,5,8\}}	\quad 
	& {B_6} &= {\{1,4,6,7\}} \\
	 {B_7} &= {\{3,4, 7,8\}} \quad
	& {B_8} &= {\{3,4,5,6\}} \\ 
	 {B_9} &= {\{2,4,6,8\}} \quad
 &{B_{10}} &= {\{2,4,5,7\}} \\ 
  {B_{11}} & = {\{2,3,6,7\}} \quad 
 & {B_{12}} &= {\{2,3,5,8\}}
\end{align*}

Suppose $A_1$ wants to repair their share. Then, the relevant cutsets are
\begin{align*}
\C_1 &= \{B_1,B_2,B_3,B_4,B_5, B_6\}\\
\C_2 &= \{B_1,B_2,B_9,B_{10},B_{11},B_{12} \}\\ 
\C_3 &= \{ B_3,B_4,B_7,B_{8},B_{11},B_{12}\}\\
\C_4 &= \{ B_5,B_6,B_7,B_{8},B_{9},B_{10}\}.
\end{align*}
\hfill $\blacksquare$
\end{example}


\begin{lemma}\label{def:repairSetAvailable}
Let $B = \{x_1,\dots, x_k\}$ be a block in the distribution design $(X, \B)$.
There exists an available repair set for $B$ if and only if 
no $\C_j$, for $1 \leq j \leq k$, fails. 
\end{lemma}

\subsection{Existence of Available Repair Sets for $t$-$(v, k, 1)$ Designs}

First, we consider Steiner quadruple systems, as a warmup. Then we generalize our
formulas to arbitrary $t$-$(v, k, 1)$ designs.

\begin{theorem}\label{thm:existSQS}
Suppose $(X,\B)$ is an $\SQS(v)$ and let $B = \{x_1,x_2,x_3,x_4\} \in \B$.
Let $q=1-p$, where $p$ is the probability that a block is available. Then 
\[\R(p) =1-4q^{r_1-1} +6q^{2r_1-r_2-1} -4q^{3r_1-3r_2} +q^{4r_1-6r_2+2},\]
where $r_1= \binom{v-1}{2}/3$ and 
$r_2= \binom{v-2}{1}/{2}$ are the replication numbers of the $\SQS$. 
\end{theorem}

\begin{proof}
From Lemma \ref{def:repairSetAvailable}, a repair set exists if no $\C_j$ fails, $1 \leq j \leq 4$. Therefore, 
\[\R(p)=1-\prob[\text{at least one } \C_j \text{ fails}].\]
For $1 \leq j \leq 4$, let $E_j$ denote the event that $\C_j$ fails.
We have
\[\prob[\text{at least one } \C_j \text{ fails}] = 
\prob[E_1 \text{ or } E_2 \text{ or } E_3 \text{ or } E_4].\]
We note the following.
\begin{enumerate}
    \item  $|\C_j|=r_1-1$  for $1 \leq j \leq 4$. Therefore,  \[\prob[E_i]=q^{r_1-1}.\] 
    \item  $|\C_j \cup \C_{j'}| = 2(r_1-1)-(r_2-1) = 2r_1-r_2-1$, for all $j,j', j \neq j'$. Therefore,  
    \[\prob[E_j \text{ and } E_{j'}] =q^{2r_1-r_2 -1}.\] 
    \item  $|\C_j \cup \C_{j'} \cup \C_{j''}|= 3(r_1-1)-3(r_2-1) = 3r_1-3r_2$, for all distinct $j,j',j''$ such that  
    $1 \leq j,j',j'' \leq 4$. Therefore, 
    \[\prob[E_j \text{ and } E_{j'} \text{ and } E_{j''}] =q^{3r_1-3r_2}.\] 
    \item  $|\C_1 \cup \C_2 \cup \C_3 \cup \C_4| = 4(r_1-1)-6(r_2-1) = 4r_1-6r_2+2$. 
    Therefore, \[ \prob[E_1 \text{ and } E_2 \text{ and } E_3 \text{ and } E_4]=q^{4r_1-6r_2+2}.\] 
\end{enumerate}
Applying the principle of inclusion-exclusion, we have 
\[\prob[E_1 \text{ or } E_2 \text{ or } E_3 \text{ or } E_4] = \binom{4}{1}q^{r_1-1} -\binom{4}{2}q^{2r_1-r_2-1} + \binom{4}{3}q^{3r_1-3r_2} - \binom{4}{4}q^{4r_1-6r_2+2}.\]
Therefore, 
\[\R(p)=1-4q^{r_1-1} + 6q^{2r_1-r_2-1} - 4q^{3r_1-3r_2} + q^{4r_1-6r_2+2}.\]
\end{proof}

The following can be proven in a similar manner.

\begin{theorem}\label{thm:genExist}
Suppose $(X,\B)$ is a  $t$-$(v, k, 1)$ design  and let $B \in \B$.
Let $q=1-p$, where $p$ is the probability that a block is available.  Then
\[ \R(p)=1- \binom{k}{1} q^{e_1} + \binom{k}{2}q^{e_2}- \binom{k}{3}q^{e_3} + \dots   + (-1)^{k+1}\binom{k}{k}q^{e_k},\]
where 
 $r_j=\frac{\binom{v-j}{t-j}}{\binom{k-j}{t-j}}$,
 for $1 \leq j \leq t$, are the replication numbers of the design, and
\[e_i = \sum_{j=1}^{\min\{i,t-1\}} (-1)^{j+1} \binom{i}{j}(r_j -1),\]
for $1 \leq i \leq k$.
\end{theorem}

\begin{proof}
We use similar notation as in the proof of Theorem \ref{thm:existSQS}.
We compute  \[\prob[\text{at least one } \C_j \text{ fails}] = 
\prob[E_1 \text{ or } E_2 \text{ or } E_3 \text{ or } E_4 \text{ or } \cdots \text{ or } E_k].\]

Let $e_i$ denote the cardinality of the union of $i$ of the sets $\C_1, \dots , \C_k$.
We will apply the principle of inclusion-exclusion to compute the values of the $e_i$'s. Note that we make use of the
fact that no block intersects $B$ in more than $t-1$ points, so the intersection of $t$ or more of the sets
$\C_1, \dots , \C_k$ is empty. Therefore, 
\begin{eqnarray*}
    e_1  &=& r_1 -1\\
    e_2  &=& 2(r_1 -1) - (r_2-1)\\
    e_3  &=& 3(r_1-1) - 3(r_2-1) + (r_3-1),
\end{eqnarray*}
etc., where no sum contains terms past $r_{t-1}$. 
In general, \[e_i = \sum_{j=1}^{\min\{i,t-1\}} (-1)^{j+1} \binom{i}{j}(r_j -1),\]
for $1 \leq i \leq k$. 

Now that we have computed $e_i$ for each $i$, we can evaluate the probability that any number 
of  $C_i$'s fail. Recall that $q=1-p$, where $p$ is the probability the player with that share is available. 
A second application of the principle of inclusion-exclusion yields the desired result:
\[\R(p) = 1- \binom{k}{1} q^{e_1} + \binom{k}{2}q^{e_2}- \binom{k}{3}q^{e_3} + \dots   
+ (-1)^{k+1}\binom{k}{k}q^{e_k}.\]
\end{proof}

\subsection{Expected Number of Minimal Repair Sets for SQS}

In general, we can determine the expected number of repair sets for a given
distribution design if we know all the minimal repair sets. The following 
formula, which is proven in the same fashion as Theorem \ref{thm:expectbibd}, can be used.

\begin{theorem}
\label{gen-exp.thm}
Suppose $(X,\B)$ is a distribution design and
let the minimal repair sets be enumerated as $\M_1, \dots , \M_s$.
Then 
\[ E(p) = \sum_{j=1}^{s} p^{|\M_i|}.\]
\end{theorem}

Of course, for an arbitrary distribution design, there can be minimal repair
sets of various sizes. For example, 
in the case of Steiner quadruple systems, minimal repair sets can be of size 
two, three or four. 

We illustrate the computation of the expected number of available minimal repair sets
on a particular design, namely, the $\SQS(10)$. 

\begin{example}
\label{SQS10.ex}
Here is the (unique) $3$-$(10,4,1)$-design:
	\begin{align*}
	 {A_0} &= {\{1, 2, 4,5\}} \quad 
	& {B_0} &= {\{1,2,3,7\}}	\quad
	 & {C_0} &= {\{1,3,5,8\}}	\\
	 {A_1} &= {\{2,3,5,6\}} \quad 
	& {B_1} &= {\{2,3,4,8\}}	\quad
	 & {C_1} &= {\{2,4,6,9\}}	\\
	 {A_2} &= {\{3,4,6,7\}} \quad 
	& {B_2} &= {\{3,4,5,9\}}	\quad
	 & {C_2} &= {\{3,5,7,0\}}	\\
	 {A_3} &= {\{4,5,7,8\}} \quad 
	& {B_3} &= {\{4,5,6,0\}}	\quad
	 & {C_3} &= {\{4,6,8,1\}}	\\
	 {A_4} &= {\{5,6,8,9\}} \quad 
	& {B_4} &= {\{5,6,7,1\}}	\quad
	 & {C_4} &= {\{5,7,9,2\}}	\\
	 {A_5} &= {\{6,7,9,0\}} \quad 
	& {B_5} &= {\{6,7,8,2\}}	\quad
	 & {C_5} &= {\{6,8,0,3\}}	\\
	 {A_6} &= {\{7,8,0,1\}} \quad 
	& {B_6} &= {\{7,8,9,3\}}	\quad
	 & {C_6} &= {\{7,9,1,4\}}	\\
	 {A_7} &= {\{8,9,1,2\}} \quad 
	& {B_7} &= {\{8,9,0,4\}}	\quad
	 & {C_7} &= {\{8,0,2,5\}}	\\
	 {A_8} &= {\{9,0,2,3\}} \quad 
	& {B_8} &= {\{9,0,1,5\}}	\quad
	 & {C_8} &= {\{9,1,3,6\}}	\\
	 {A_9} &= {\{0,1,3,4\}} \quad 
	& {B_9} &= {\{0,1,2,6\}}	\quad
	 & {C_9} &= {\{0,2,4,7\}}	
\end{align*}

Suppose we want to repair the block $A_0 = \{1,2,4,5\}$.
We consider minimal repair sets of sizes $2,3$ and $4$ in turn.
All the  computations are applications of Theorem \ref{gen-exp.thm}.

A repair set of size two consists of 
\begin{itemize}
\item a block containing ${1,2}$ and a block containing ${4,5}$; or 
\item a block containing ${1,4}$ and a block containing ${2,5}$; or 
\item a block containing ${1,5}$ and a block containing ${2,4}$.
\end{itemize}
The total number of choices for these two blocks is $3 \times 3 \times 3 = 27$
(there are three subcases, and in each subcase there are three choices of each 
of the two blocks).
Therefore, the expected number of minimal repair sets of size two is  $27p^2$.

A minimal repair set of size four consists of four blocks having the 
following form:
\begin{itemize}
\item a block  containing ${1}$, but none of ${2,4,5}$
\item a block  containing ${2}$, but none of ${1,4,5}$
\item a block  containing ${4}$, but none of ${1,2,5}$
\item a block  containing ${5}$, but none of ${1,2,4}$.
\end{itemize}
There are {two choices} for each of these four blocks, so the total number of choices is $2^4 = 16$.
Therefore, the expected number of minimal repair sets of size four is ${16p^4}$.

A {minimal} repair set of size three can have three possible 
forms:
\begin{description}
\item [type {pair-pair-pair}:] three pairs intersecting in a point,
e.g.,  ${12, 14, 15}$. There are {four} configurations of this type.
 \item [type {pair-pair-point}:] two pairs intersecting in a point,
 and a disjoint point
e.g., ${12, 14, 5}$. There are {twelve} configurations of this type.
 \item [type {pair-point-point}:] one pair,
 and two disjoint points, 
e.g., ${12, 4, 5}$. There are {six} configurations of this type.
\end{description}
After some counting, the expected number of minimal repair sets of size three is 
seen to be
\[ ({4} \: {\times \: 3^3} + {12} \: {\times \: 3^2 \times 2} + {6} \: {\times \: 3 \times 2^2}) {p^3} = {396 p^3}.\]
 Finally, we have 
\[ E(p) = 27p^2 + 396 p^3 + 16p^4.\]
\hfill $\blacksquare$
\end{example}

The general case of an $\SQS(v)$ is similar. We tabulate the expected number of minimal repair sets of the various
types in the Table \ref{tab1}.

\begin{table}
\caption{Expected number of repair sets for $\SQS(v)$}
\label{tab1}
\begin{center}
\begin{tabular}{c|c|c}
size & type & expected number \\ \hline
2 & & $3(r_2-1)^2 p^2$\\
3 & pair-pair-pair & $4(r_2-1)^3 p^3$ \\
3 & pair-pair-point & $12(r_2-1)^2(r_1-3r_2+2)  p^3$ \\
3 & pair-point-point & $6(r_2-1)(r_1-3r_2+2)^2 p^3$ \\
4 & & $(r_1-3r_2+2)^4 p^4$
\end{tabular}
\end{center} 
\end{table}

We can now combine the expected number of repair sets for each size and type to 
produce the expected number of repair sets, which we record in the following theorem.

\begin{theorem}\label{thm:expectSQS}
Suppose $(X,\B)$ is an $\SQS(v)$ and let $B \in \B$.
Let $q=1-p$, where $p$ is the probability that a block is available. 
Then   
\[ E(p) = 3(r_2-1)^2 p^2 + 2 (r_2 - 1) (3 {r_1}^2 - 12 r_1 r_2 + 6 r_1 + 11 {r_2}^2 - 10 r_2 + 2) p^3 + (r_1-3r_2+2)^4 p^4.\]
\end{theorem}

\section{Discussion and Summary}
\label{sum.sec}

We have introduced the problem of studying reliability of combinatorial
RTS. We employed techniques from network reliability theory to aid in 
the derivation of some of our formulas. Perhaps this approach will prove
useful in other combinatorial design problems that can be phrased
in terms of network reliability.

The material in this paper is from the Masters Thesis of the first author \cite{thesis}.
The thesis \cite{thesis} also studies efficient algorithms to actually find
a repair set for various kinds of distribution designs.

\bibliographystyle{amsplain}



\begin{thebibliography}{XX}

\bibitem{BL}
J.\ Benaloh and J.\ Leichter.
Generalized secret sharing and monotone functions.
{\it Lecture Notes in Computer Science} {\bf 403} (1990), 27--35
(CRYPTO '88 Proceedings).

\bibitem{colbourn}
C.J.\ Colbourn. 
{\it The Combinatorics of Network Reliability}, Oxford University Press, 1987. 

\bibitem{HCD}
C.J.\ Colbourn and J.H.\ Dinitz.
{\it Handbook of Combinatorial Designs, Second Edition.}
Chapman \& Hall/CRC, 2006.

\bibitem{thesis}
Bailey Kacsmar. 
{\it Designing Efficient Algorithms for
Combinatorial Repairable Threshold
Schemes}. Masters Thesis, University of Waterloo, 2018.




\bibitem{laing-stinson}
T.M.\ Laing and D.R.\ Stinson.
A survey and refinement of repairable threshold schemes.
\emph{Journal of Mathematical Cryptology} \textbf{12} (2018), 57--81.






  
  
\bibitem{Shamir} A.\ Shamir.  How to share a secret. {\it Communications of the ACM} {\bf 22} (1979), 612--613.



\bibitem{CTAP} D.R.\ Stinson and M.B.\ Paterson. {\it Cryptography Theory and Practice, Fourth Edition}. Chapman \& Hall/CRC, Boca Raton, 2019.

\bibitem{stinson-wei}
D.R.\ Stinson and R.\ Wei. 
Combinatorial repairability for threshold schemes.
\emph{Designs, Codes and Cryptography} \textbf{86} (2018), 195--210.

\end{thebibliography}

\end{document}